\documentclass{birkjour}
 \newtheorem{thm}{Theorem}[section]
 
 \newtheorem{lem}[thm]{Lemma}
 
 \theoremstyle{definition}
 \newtheorem{defn}[thm]{Definition}
 \theoremstyle{remark}
 \newtheorem{rem}[thm]{Remark}
 \newtheorem*{ex}{Example}
 \numberwithin{equation}{section}

\def\qfrac#1#2{\left[#1\atop #2\right]}

\newcommand{\loc}{\mathop{\mathrm{loc}}}

\newcommand{\D}{\mathrm{d}}

\begin{document}

%
%
%
%
%
%
%
%
%

\title[$(p,q)$-Beta Functions and Applications]
 {$(p,q)$-Beta Functions and Applications\\ in Approximation}

\author[G.V. Milovanovi\'c]{Gradimir V. Milovanovi\'c}

\address{%
Serbian Academy of Sciences and Arts\\
Kneza Mihaila 35, 11000 Beograd, Serbia\\
\& State University of Novi Pazar\\
Novi Pazar, Serbia}

\email{gvm@mi.sanu.ac.rs}

\thanks{This paper was supported
 by the Serbian Ministry of Education, Science and Technological Development (No. \#OI\,174015).}
\author{Vijay Gupta}
\address{Department of Mathematics\br
Netaji Subhas Institute of Technology\br
Sector 3 Dwarka, New Delhi-110078, India}
\email{vijaygupta2001@hotmail.com}

\author{Neha Malik}
\address{Department of Mathematics\br
Netaji Subhas Institute of Technology\br
Sector 3 Dwarka, New Delhi-110078, India}
\email{neha.malik\_nm@yahoo.com}



\date{January 1, 2016}

\begin{abstract}
In the present paper, we consider $(p,q)$-analogue of the Beta operators and using it, we propose the integral modification of the generalized Bernstein polynomials. We estimate some direct results on local and global approximation. Also, we illustrate some graphs for the convergence of $(p,q)$-Bernstein-Durrmeyer operators for different values of the parameters $p$ and $q$ using {\sc Mathematica} package.\end{abstract}

\maketitle
\section{Introduction}

In the last decade, the generalizations of several operators to quantum variant have been introduced and their approximation behavior have been discussed [see for instance \cite{avr}, \cite{amc}, \cite{ZFVG}, \cite{GVM}, \cite{vgrpa} etc]. The further generalization of quantum calculus is the post-quantum calculus, denoted by $(p,q)$-calculus. Recently, some researchers started working in this direction (cf. \cite{TA}, \cite{vgumi}, \cite{4}, \cite{VSSY}). Some basic definitions and theorems, which are mentioned below may be found in these papers and references therein.
\begin{equation*}
\left[ n\right] _{p,q}:=p^{n-1}+p^{n-2}q+p^{n-3}q^2+\cdots +pq^{n-2}+q^{n-1}=\frac{p^{n}-q^{n}}{p-q} \cdot
\end{equation*}
The $\left( p,q\right)$-factorial is given by
$\left[ n\right] _{p,q}!=\prod\limits_{k=1}^{n}\left[ k\right] _{p,q},
\ \ n\ge  1, \ \ \left[0\right] _{p,q}!=1$.
The $\left( p,q\right) $-binomial coefficient satisfies%
\begin{equation*}
\qfrac{n}{k}_{p,q}=\frac{\left[ n\right]_{p,q}!}{\left[ n-k\right] _{p,q}!\left[
k\right] _{p,q}!}, \ \ 0\le  k\le  n.
\end{equation*}

The $(p,q)$-power basis is defined as
\begin{eqnarray*}
(x\ominus a)_{p,q}^{n}&=&(x-a)(px-qa)(p^2x-q^2a)\cdots (p^{n-1}x-q^{n-1}a).
\end{eqnarray*}

The $(p,q)$-derivative of the function $f$ is defined as
\begin{equation*}
D_{p,q}f\left( x\right) =\frac{f\left( px\right) -f\left( qx\right) }{\left(
p-q\right) x}, \ x\neq 0.
\end{equation*}
As a special case when $p=1$, the $(p,q)$-derivative
reduces to the $q$-derivative.

The $(p,q)$-derivative fulfils the following product rules%
\begin{eqnarray*}
D_{p,q}(f(x)g(x)) &=&f(px)D_{p,q}g(x)+g(qx)D_{p,q}f(x), \\[2mm]
D_{p,q}(f(x)g(x)) &=&g(px)D_{p,q}f(x)+f(qx)D_{p,q}g(x).
\end{eqnarray*}

Obviously $D_{p,q}(x\ominus a)_{p,q}^{0}=0$ and for $n\ge  1,$ we have
\begin{eqnarray*}
D_{p,q}(x\ominus a)_{p,q}^{n} &=&\left[ n\right] _{p,q}(px\ominus
a)_{p,q}^{n-1}, \\[2mm]
D_{p,q}(a\ominus x)_{p,q}^{n} &=&-\left[ n\right] _{p,q}(a\ominus
qx)_{p,q}^{n-1}.
\end{eqnarray*}

Let $f$ be an arbitrary function and $a\in \mathbb{R}$. The $(p,q)$-integral of $f(x)$ on \ $\left[ 0,a\right] $ (see \cite{pns}) is defined as
\begin{equation*}
\int\limits_{0}^{a}f\left( x\right) d_{p,q}x=\left( q-p\right) a\sum
\limits_{k=0}^{\infty }\frac{p^{k}}{q^{k+1}}f\left( \frac{p^{k}}{q^{k+1}}%
a\right) \  \  \mbox{ if } \  \  \left \vert \frac{p}{q}\right \vert <1
\end{equation*}
and
\begin{equation*}
\int\limits_{0}^{a}f\left( x\right) d_{p,q}x=\left( p-q\right) a\sum
\limits_{k=0}^{\infty }\frac{q^{k}}{p^{k+1}}f\left( \frac{q^{k}}{p^{k+1}}%
a\right) \  \  \mbox{ if } \  \  \left \vert \frac{q}{p}\right \vert <1.
\end{equation*}

The formula of $(p,q)$-integration by part is given by
\begin{eqnarray}\label{a1}
\int\limits_{a}^{b}f\left( px\right) D_{p,q}g\left( x\right) {\D}_{p,q}x&=&f\left(
b\right) g\left( b\right) -f\left( a\right) g\left( a\right)\nonumber\\
&&\qquad -\int\limits_{a}^{b}g\left( qx\right) D_{p,q}f\left( x\right){\D}_{p,q}x	\end{eqnarray}

Very recently Gupta and Aral \cite{VA} proposed the $(p,q)$ analogue of usual Durrmeyer operators by considering some other form of $(p,q)$ Beta functions, which is not commutative.
In the present article, we define different $(p,q)$-variant of Beta function of first kind and find an identity relation with $(p,q)$-Gamma functions. It is observed that $(p,q)$-Beta functions may satisfy the commutative property, by multiplying the appropriate factor while choosing $(p,q)$ Beta function. As far as the approximation is concerned, order is important in post-quantum calculus. We propose a generalization of Durrmeyer type operators and establish some direct results.

\section{$(p,q)$-Gamma and $(p,q)$-Beta Functions}

\begin{defn}[\cite{2}]
 Let $n$ is a nonnegative integer, we define the $(p,q)$-Gamma
function as%
\begin{equation*}
\Gamma _{p,q}\left( n+1\right) =\frac{(p\ominus q)_{p,q}^{n}}{(p-q)^{n}}=%
\left[ n\right] _{p,q}!, \quad  \ 0<q<p.
\end{equation*}
\end{defn}

\begin{defn}
Let $m,n\in \mathbb{N},$ we define $\left( p,q\right) $-Beta integral as%
\begin{equation}
B_{p,q}\left( m,n\right) =\int\limits_{0}^{1}x
^{m-1}\left( 1\ominus qx\right) _{p,q}^{n-1}{\D}_{p,q}x.  \label{a5}
\end{equation}
\end{defn}

\begin{thm}
The $(p,q)$-Gamma and $(p,q)$-Beta functions fulfil the following
fundamental relation%
\begin{equation}
B_{p,q}\left( m,n\right) =p^{(n-1)(2m+n-2)/2}\frac{\Gamma
_{p,q}\left( m\right) \Gamma _{p,q}\left( n\right) }{\Gamma _{p,q}\left(
m+n\right) },  \label{a4}
\end{equation}%
where $m,n\in \mathbb{N}$.
\end{thm}

\begin{proof}
For any $m,n\in $ $\mathbb{N}$ since%
\begin{equation*}
B_{p,q}\left( m,n\right) =\int\limits_{0}^{1}x
^{m-1}\left( 1\ominus qx\right) _{p,q}^{n-1}{\D}_{p,q}x,
\end{equation*}%
using (\ref{a1}) for $f\left( x\right) =\left({x}/{p}\right)^{m-1}$ and $g\left( x\right) =-
 {\left( 1\ominus x\right) _{p,q}^{n}}/{\left[ n\right] _{p,q}}$ with
the equality $D_{p,q}\left( 1\ominus x\right) ^{n}=-\left[ n\right]
_{p,q}\left( 1\ominus qx\right) ^{n-1}$ we have
\begin{eqnarray}
B_{p,q}\left( m,n\right) &=&\frac{\left[ m-1\right] _{p,q}}{p^{m-1}\left[ n%
\right] _{p,q}}\int\limits_{0}^{1}x^{m-2}\left( 1\ominus
qx\right) _{p,q}^{n}{\D}_{p,q}x  \notag \\
&=&\frac{\left[ m-1\right] _{p,q}}{p^{m-1}\left[ n\right] _{p,q}}%
B_{p,q}\left( m-1,n+1\right).  \label{a2}
\end{eqnarray}%
Also we can write for positive integer $n$%
\begin{eqnarray*}
B_{p,q}\left( m,n+1\right) \!\!\!&=&\!\!\!\int\limits_{0}^{1}x
^{m-1}\left( 1\ominus qx\right) _{p,q}^{n}{\D}_{p,q}x \\
\!\!\!&=&\!\!\!\int\limits_{0}^{1}x^{m-1}\left( 1\ominus qx\right)
_{p,q}^{n-1}\left( p^{n-1}-q^{n}x\right) {\D}_{p,q}x \\
\!\!\!&=&\!\!\!p^{n-1}\int\limits_{0}^{1}x^{m-1}\left( 1\ominus
qx\right) _{p,q}^{n-1}{\D}_{p,q}x -q^{n}\int\limits_{0}^{1}x^{m}\left(1\ominus
qx\right) _{p,q}^{n-1}{\D}_{p,q}x \\
\!\!\!&=&\!\!\!p^{n-1}B_{p,q}\left( m,n\right) -q^{n}B_{p,q}\left( m+1,n\right) .
\end{eqnarray*}%
Using (\ref{a2}), we have
\begin{equation*}
B_{p,q}\left( m,n+1\right) =p^{n-1}B_{p,q}\left( m,n\right) -q^{n}\frac{\left[
m\right] _{p,q}}{p^{m}\left[ n\right] _{p,q}}B_{p,q}\left( m,n+1\right),
\end{equation*}%
which implies that
\begin{equation*}
B_{p,q}\left( m,n+1\right) =p^{n+m-1}\frac{p^{n}-q^{n}}{p^{n+m}-q^{n+m}}%
B_{p,q}\left( m,n\right) .
\end{equation*}%
Further, by definition of $(p,q)$ integration
\begin{eqnarray*}
B_{p,q}\left( m,1\right) &=&\int\limits_{0}^{1}x
^{m-1}{\D}_{p,q}x=\frac{1}{\left[ m\right] _{p,q}} \cdot
\end{eqnarray*}%
We immediately have%
{\small\begin{eqnarray}
B_{p,q}\left( m,n\right) \!\!\!&=&\!\!\!p^{n+m-2}\frac{p^{n-1}-q^{n-1}}{%
p^{n+m-1}-q^{n+m-1}}B_{p,q}\left( m,n-1\right)  \notag \\
\!\!\!&=&\!\!\!p^{n+m-2}\frac{p^{n-1}-q^{n-1}}{p^{n+m-1}-q^{n+m-1}}\,p^{n+m-3}\frac{%
p^{n-2}-q^{n-2}}{p^{n+m-2}-q^{n+m-2}}B_{p,q}\left( m,n-2\right)  \notag \\
\!\!\!&=&\!\!\!p^{n+m-2}\frac{p^{n-1}-q^{n-1}}{p^{n+m-1}-q^{n+m-1}}\,p^{n+m-3}\frac{%
p^{n-2}-q^{n-2}}{p^{n+m-2}-q^{n+m-2}}\cdots \notag\\
&&\qquad\qquad \qquad\qquad\qquad\qquad\times\, p^{m}\frac{p-q}{p^{m+1}-q^{m+1}}%
B_{p,q}\left( m,1\right)  \notag \\
\!\!\!&=&\!\!\!p^{n+m-2}\frac{p^{n-1}-q^{n-1}}{p^{n+m-1}-q^{n+m-1}}\,p^{n+m-3}\frac{%
p^{n-2}-q^{n-2}}{p^{n+m-2}-q^{n+m-2}}\cdots\notag\\
&&\qquad\qquad \qquad\qquad\qquad\qquad\times\, p^{m}\frac{p-q}{p^{m+1}-q^{m+1}}%
\frac{1}{\left[ m\right] _{p,q}}  \notag \\
\!\!\!&=&\!\!\!p^{m+(m+1)+\cdots+(m+n-2)}\frac{\left( p\ominus q\right)
_{p,q}^{n-1}}{\left( p^{m}\ominus q^{m}\right) _{p,q}^{n}}\left( p-q\right)
\notag \\
\!\!\!&=&\!\!\!p^{(n-1)(2m+n-2)/2}\frac{\left( p\ominus q\right)
_{p,q}^{n-1}}{\left( p^{m}\ominus q^{m}\right) _{p,q}^{n}}\left( p-q\right).
\label{a3}
\end{eqnarray}}

Following \cite{2}, we have $(a\ominus b)_{p,q}^{n+m}=(a\ominus
b)_{p,q}^n(ap^n\ominus bq^n)_{p,q}^m$ thus (\ref{a3}) leads to

{\small\begin{eqnarray*}
B_{p,q}\left( m,n\right) \!\!\!\!&=&\!\!\!\!p^{\frac{(n-1)(2m+n-2)}2}\frac{\left(
p\ominus q\right) _{p,q}^{n-1}}{\left( p^{m}\ominus q^{m}\right) _{p,q}^{n}}%
\left( p-q\right) \\
\!\!\!\!&=&\!\!\!\!p^{\frac{(n-1)(2m+n-2)}2}\frac{\left( p\ominus q\right)
_{p,q}^{n-1}}{\left( p-q\right) ^{n-1}}\cdot\frac{\left( p\ominus q\right)
_{p,q}^{m-1}}{\left( p-q\right) ^{m-1}}\cdot\frac{\left( p-q\right) ^{m-1}\left(
p-q\right) ^{n-1}}{\left( p\ominus q\right) _{p,q}^{m-1}\left( p^{m}\ominus
q^{m}\right) _{p,q}^{n}}\left( p-q\right) \\
\!\!\!\!&=&\!\!\!\!p^{\frac{(n-1)(2m+n-2)}2}\frac{\left( p\ominus q\right)
_{p,q}^{n-1}}{\left( p-q\right) ^{n-1}}\cdot\frac{\left( p\ominus q\right)
_{p,q}^{m-1}}{\left( p-q\right) ^{m-1}}\cdot\frac{\left( p-q\right) ^{m+n-1}}{%
\left( p\ominus q\right) _{p,q}^{m+n-1}} \\
\!\!\!\!&=&\!\!\!\!p^{\frac{(n-1)(2m+n-2)}2}\frac{\Gamma _{p,q}\left( m\right)
\Gamma _{p,q}\left( n\right) }{\Gamma _{p,q}\left( m+n\right) } \cdot
\end{eqnarray*}}
This completes the proof of the theorem.
\end{proof}

\begin{rem} The following observations have been made for $(p,q)$ Beta functions:

\begin{itemize}
 \item For $m,n\in \mathbb{N}$, we have
$$B_{p,q}(m,n+1)=p^{n-1}B_{p,q}(m,n)-q^nB_{p,q}(m+1,n).$$
\item
The $\left( p,q\right) $-Beta integral defined by (\ref{a5}) is not commutative. In order to make commutative, we may consider the following form
\begin{equation*}
\widetilde{B}_{p,q}\left( m,n\right) =\int\limits_{0}^{1}p^{m(m-1)/2}x
^{m-1}\left( 1\ominus qx\right) _{p,q}^{n-1}{\D}_{p,q}x.
\end{equation*}
For this form, $(p,q)$-Gamma and $(p,q)$-Beta functions fulfill the following
fundamental relation
\begin{equation}
\widetilde{B}_{p,q}\left( m,n\right) =p^{(2mn+m^2+n^2-3m-3n+2)/2}\frac{\Gamma
_{p,q}\left( m\right) \Gamma _{p,q}\left( n\right) }{\Gamma _{p,q}\left(
m+n\right) },  \label{a-com}
\end{equation}%
where $m,n\in \mathbb{N}.$ Obviously for form (\ref{a-com}), we get $\widetilde{B}_{p,q}\left( m,n\right)=\widetilde{B}_{p,q}\left( n,m\right).$
\end{itemize}
\end{rem}

\section{$(p,q)$ Bernstein-Type Operators and Moments}

For $n\in \mathbb{N}$ and $k\ge  0,$ we have the following identity, which can be easily verified using the principle of mathematical induction:
\begin{equation}\label{eb0}\sum\limits_{k=0}^n \qfrac{n}{k}_{p,q}p^{k(k-1)/2} x^{k}(1\ominus x)_{p,q}^{n-k}=p^{n(n-1)/2}.
\end{equation}
Using the above identity, we consider the $\left( p,q\right)$-analogue of Bernstein operators for $x\in \left[ 0,1\right]$ and $0<q<p\le  1$ as
\begin{equation}\label{eb}
B_{n,p,q}\left( f,x\right) =\sum\limits_{k=0}^{n}b_{n,k}^{p,q}(1,x)f\left(\frac{%
p^{n-k}[k]_{p,q}}{[n]_{p,q}}\right) ,\end{equation}
where the $(p,q)$-Bernstein basis is defined as
\begin{equation*}
b_{n,k}^{p,q}(1,x)=\qfrac{n}{k}_{p,q}p^{[k(k-1)-n(n-1)]/2} x^{k}(1\ominus x)_{p,q}^{n-k} .
\end{equation*}

\begin{rem} \label{r0}
Other form of the $(p,q)$-analogue of Bernstein polynomials has been recently considered by Mursaleen et al. \cite{4}.
\end{rem}

\begin{rem} \label{r1}
Using \ the \ identity  \ (\ref{eb0}) \ and \ the \ following \ recurrence \ relation (for $m\ge  1,\
U_{n,m}^{p,q}(x)=B_{n,p,q}(e_{m},x)=B_{n,p,q}(t^{m},x))$:
\[ [n]_{p,q} U_{n,m+1}^{p,q}(px)=p^{n} x (1-px) D_{p,q}[U_{n,m}^{p,q}(x)] + [n]_{p,q} px U_{n,m}^{p,q}(px), \]
 the $(p,q)$-Bernstein polynomial satisfy
\[B_{n,p,q}\left(e_{0},x\right)=1,\quad
B_{n,p,q}\left(e_{1},x\right)=x,\quad
B_{n,p,q}\left(e_{2},x\right)=x^2+\frac{p^{n-1}x(1-x)}{[n]_{p,q}},\]
where $e_{i}=t^{i},$ $i=0,1,2.$
\end{rem}

Recently, Gupta and Wang (see \cite{vghw}) discussed the $q$-variant of certain Bern\-stein-Durrmeyer type operators. We now extend these studies and propose the following $(p,q)$-Bernstein-Durrmeyer operators based on $(p,q)$-Beta function.

For $x\in[0,1]$ and $0<q<p\le  1,$ the $\left( p,q\right)$-analogue of Bernstein-Durrmeyer operators is defined as
\begin{eqnarray} \label{operator}
D_{n}^{p,q}(f,x)&=&\left[ n+1\right] _{p,q}
\sum\limits_{k=1}^{n}p^{-(n-k+1)(n+k)/2}b_{n,k}^{p,q}(1,x)\nonumber\\[2mm]	
&&\qquad\quad\times\,\int\limits_{0}^{1}b_{n,k-1}^{p,q}(t)f( t){\D}_{p,q}t + b_{n,0}^{p,q}(1,x)f(0),
\end{eqnarray}	
where $b_{n,k}^{p,q}(1,x)$ is defined by (\ref{eb}) and
$$b_{n,k}^{p,q}(t)=\qfrac{n}{k}_{p,q} t^{k}(1\ominus qt)_{p,q}^{n-k}.$$

It may be remarked here that for $p=1,$ these operators will not reduce to the $q$-Durrmeyer operators; but for $p=q=1,$ these will reduce to the Durrmeyer operators.

\begin{lem}\label{l1}
Let $e_{m}=t^{m},$ $m \in \mathbb{N}\cup \{0\},$ then for $x\in \lbrack 0,1]$ and $0<q<p\le  1,$ we have
\begin{equation*}
D_{n}^{p,q}(e_{0},x)=1,\quad
D_{n}^{p,q}(e_{1},x)=\frac{p[n]_{p,q}x}{[n+2]_{p,q}},
\end{equation*}

\begin{equation*}
D_{n}^{p,q}(e_{2},x)=\frac{(p+q)p^{n+1}[n]_{p,q}x}{[n+2]_{p,q}[n+3]_{p,q}}+\frac{([n]_{p,q}-p^{n-1})p^{2}q[n]_{p,q}x^2}{[n+2]_{p,q}[n+3]_{p,q}} \cdot
\end{equation*}
\end{lem}

\begin{proof}
Using (\ref{a4}) and (\ref{a5}) and Remark \ref{r1}, we have
\begin{eqnarray*}
D_{n}^{p,q}(e_{0},x) &=&\left[ n+1\right] _{p,q}%
\sum\limits_{k=1}^{n}p^{-[(n+1-k)( n+k)/2]} b_{n,k}^{p,q}(1,x)\\
&&\quad\times\,\int\limits_{0}^{1}\qfrac{n}{k-1}_{p,q} t^{k-1} \, (1\ominus qt)_{p,q}^{n+1-k}{\D}_{p,q}t+b_{n,0}^{p,q}(1,x) \\
&=&\left[ n+1\right] _{p,q}%
\sum\limits_{k=1}^{n}p^{-[(n+1-k)( n+k)/2]} b_{n,k}^{p,q}(1,x)\qfrac{n}{k-1}\\
&&\quad\times\, B_{p,q}(k,n-k+2)+b_{n,0}^{p,q}(1,x) \\
&=&\left[ n+1\right] _{p,q}%
\sum\limits_{k=1}^{n}p^{-[(n+1-k)( n+k)/2]}  b_{n,k}^{p,q}(1,x) \frac{[n]_{p,q}![k-1]_{p,q}!}{[n+1-k]_{p,q}!}\\
&& \quad\times \  p^{[(n+1-k)( n+k)/2]} \
\frac{[k-1]_{p,q}![n-k+1]_{p,q}!}{[n+1]_{p,q}!%
}+b_{n,0}^{p,q}(1,x)  \\
&=&B_{n,p,q}\left( 1,x\right) =1.
\end{eqnarray*}
Next, applying Remark \ref{r1}, we have
\begin{eqnarray*}
D_{n}^{p,q}(e_{1},x) &=&\left[ n+1\right] _{p,q}%
\sum\limits_{k=1}^{n}p^{-[(n+1-k)(n+k)/2]}  b_{n,k}^{p,q}(1,x)\\
&&\quad\times\,\int\limits_{0}^{1}\qfrac{n}{k-1}_{p,q} t^{k} \,
(1\ominus qt)_{p,q}^{n+1-k} {\D}_{p,q}t \\
&=&\left[ n+1\right] _{p,q}%
\sum\limits_{k=1}^{n}p^{-[(n+1-k)( n+k)/2]}  b_{n,k}^{p,q}(1,x)\\
&&\quad\times\,\qfrac{n}{k-1}_{p,q} B_{p,q}(k+1,n-k+2) \\
&=&\left[ n+1\right] _{p,q}%
\sum\limits_{k=1}^{n}p^{-[(n+1-k)( n+k)/2]} \ b_{n,k}^{p,q}(1,x)\qfrac{n}{k-1}_{p,q}\\
&& \quad\times \ p^{(n+1-k)(n+k+2)/2} \ \frac{[k]_{p,q}![n-k+1]_{p,q}!}{%
[n+2]_{p,q}!} \\
&=&\sum\limits_{k=1}^{n}p^{n-k+1}b_{n,k}^{p,q}(1,x) \ \frac{[k]_{p,q}}{[n+2]_{p,q}} \\
&=&\frac{p[n]_{p,q}}{[n+2]_{p,q}}%
\sum\limits_{k=1}^{n}b_{n,k}^{p,q}(1,x)\frac{p^{n-k}[k]_{p,q}}{[n]_{p,q}}=\frac{p[n]_{p,q}x}{[n+2]_{p,q}} \cdot
\end{eqnarray*}

Further, using the identity $[k+1]_{p,q}=p^k+q[k]_{p,q}$ and by
Remark \ref{r1}, we get
\begin{eqnarray*}
D_{n}^{p,q}(e_{2},x) &=&\left[ n+1\right] _{p,q}%
\sum\limits_{k=1}^{n}p^{-[(n+1-k)( n+k)/2]} \ b_{n,k}^{p,q}(1,x)\\
&&\quad\times\,\int\limits_{0}^{1}\qfrac{n}{k-1}_{p,q}t^{k+1} \, (1\ominus qt)_{p,q}^{n+1-k} {\D}_{p,q}t \\
&=&\left[ n+1\right] _{p,q}%
\sum\limits_{k=1}^{n}p^{-[(n+1-k)( n+k)/2]} \ b_{n,k}^{p,q}(1,x)\\
&&\quad\times\,\qfrac{n}{k-1}_{p,q} B_{p,q}(k+2,n-k+2) \\
&=&\left[ n+1\right] _{p,q}%
\sum\limits_{k=1}^{n}p^{-[(n+1-k)( n+k)/2]} \ b_{n,k}^{p,q}(1,x)\qfrac{n}{k-1}_{p,q}\\
&& \quad\times \, p^{(n+1-k)(n+k+4)/2}\frac{[k+1]_{p,q}![n-k+1]_{p,q}!}{%
[n+3]_{p,q}!}\\
&=&\sum\limits_{k=1}^{n}p^{2(n-k+1)} \ b_{n,k}^{p,q}(1,x) \, \frac{[k]_{p,q}[k+1]_{p,q}}{[n+2]_{p,q}[n+3]_{p,q}}
\end{eqnarray*}
i.e.,
{\small\begin{eqnarray*}
D_{n}^{p,q}(e_{2},x)
&=&\sum\limits_{k=1}^{n}p^{2(n-k+1)} \ b_{n,k}^{p,q}(1,x) \, \frac{[k]_{p,q}(p^k+q[k]_{p,q})}{[n+2]_{p,q}[n+3]_{p,q}} \\
&=&\frac{p^{n+2}[n]_{p,q}}{[n+2]_{p,q}[n+3]_{p,q}}%
\sum\limits_{k=1}^{n}b_{n,k}^{p,q}(1,x) \ \frac{p^{n-k}[k]_{p,q}}{[n]_{p,q}}\\
&&
\qquad+\frac{p^{2}q[n]_{p,q}^2}{[n+2]_{p,q}[n+3]_{p,q}}\sum\limits_{k=1}^{n}b_{n,k}^{p,q}(1,x)\left(\frac{p^{n-k}[k]_{p,q}}{[n]_{p,q}}\right)^2\\
&=&\frac{p^{n+2}[n]_{p,q}x}{[n+2]_{p,q}[n+3]_{p,q}}+\frac{p^{2}q[n]_{p,q}^2}{[n+2]_{p,q}[n+3]_{p,q}}\left(x^2+\frac{p^{n-1}x(1-x)}{[n]_{p,q}}\right)\\
&=&\frac{p^{n+2}[n]_{p,q}x}{[n+2]_{p,q}[n+3]_{p,q}}+\frac{p^{2}q[n]_{p,q}^2x^2}{[n+2]_{p,q}[n+3]_{p,q}}+\frac{p^{n+1}q[n]_{p,q}x(1-x)}{[n+2]_{p,q}
[n+3]_{p,q}}\\
&=&\frac{(p+q)p^{n+1}[n]_{p,q}x}{[n+2]_{p,q}[n+3]_{p,q}}+\frac{([n]_{p,q}-p^{n-1})p^{2}q[n]_{p,q}x^2}{[n+2]_{p,q}[n+3]_{p,q}} \cdot
\end{eqnarray*}}
\end{proof}

\begin{rem} \label{cm} Using above lemma, we can obtain the following central moments:
\begin{eqnarray*}
1^\circ\!\!\!&&\!\!\! D_{n}^{p,q}\left((t-x),x\right)=\frac{(p[n]_{p,q}-[n+2]_{p,q})x}{[n+2]_{p,q}}\\[2mm]
2^\circ\!\!\!&&\!\!\! D_{n}^{p,q}\left((t-x)^2,x\right)=\frac{(p+q)p^{n+1}[n]_{p,q}x}{[n+2]_{p,q}[n+3]_{p,q}}\\[2mm]
&&\quad +\frac{[([n]_{p,q}-p^{n-1})p^{2}q[n]_{p,q}-2p[n]_{p,q}[n+3]_{p,q}+[n+2]_{p,q}
    [n+3]_{p,q}]x^2}{[n+2]_{p,q}[n+3]_{p,q}}.
\end{eqnarray*}
\end{rem}

\begin{lem}
\label{l2} Let $n$ be a given natural number, then
\begin{equation*}
D_{n}^{p,q}((t-x)^{2},x) \  \le  \  \frac{6}{[n+2]_{p,q}} \  \left(
\varphi^{2}(x) + \frac{1}{[n+2]_{p,q}} \right),
\end{equation*}
\noindent where $\varphi^{2}(x) = x ( 1 - x ),$ $x \in [0,1].$
\end{lem}

\begin{proof}
In view of Lemma \ref{l1}, we obtain
\begin{eqnarray*}
&&D_{n}^{p,q}((t-x)^{2},x)=\frac{(p+q)p^{n+1}[n]_{p,q}x}{[n+2]_{p,q}[n+3]_{p,q}}\\[2mm]
&&\quad +\frac{[([n]_{p,q}-p^{n-1})p^{2}q[n]_{p,q}-2p[n]_{p,q}[n+3]_{p,q}+[n+2]_{p,q}
    [n+3]_{p,q}]x^2}{[n+2]_{p,q}[n+3]_{p,q}}.\qquad
\end{eqnarray*}

By direct computations, using the definition of the $(p,q)$%
-numbers, we get
\[
(p+q)p^{n+1}[n]_{p,q}= p^{n+1} (p + q)(p^{n-1}+p^{n-2}q+p^{n-3}q^2+\cdots +pq^{n-2}+q^{n-1})>0
\]
for every $q \in (q_{0},1)$.

Furthermore, the expression
\[(p+q)p^{n+1}[n]_{p,q}+([n]_{p,q}-p^{n-1})p^{2}q[n]_{p,q}-2p[n]_{p,q}[n+3]_{p,q}+[n+2]_{p,q}
    [n+3]_{p,q}\]
is equal to
\begin{eqnarray*}
&&(p+q)p^{n+1}[n]_{p,q}+   p^{2}q[n]_{p,q}^2-p^{n+1}q[n]_{p,q}\\[2mm]
&&-2p[n]_{p,q}(p^{n+2}+qp^{n+1}+q^2p^n+q^3[n]_{p,q})\\[2mm]
&&+(p^{n+1}+qp^n+q^2[n]_{p,q})(p^{n+2}+qp^{n+1}+q^2p^n+q^3[n]_{p,q})\le 6.
\end{eqnarray*}
\noindent In conclusion, for $x \in [0,1],$ we have
\begin{eqnarray}
D_{n}^{p,q}((t-x)^{2},x)&\le & \frac{6}{[n+2]_{p,q}} \ \delta_n^2(x),
\end{eqnarray}
\noindent which was to be proved.
\end{proof}

\section{Local and Global Estimates}
In this section, we estimate some direct results, viz., local and global approximation in terms of modulus of continuity.

Our first main result is a local theorem.\\
For this, we denote $W^{2}=\bigl\{g\in C[0,1]: g''\in
C[0,1]\bigr\},$ for $\delta >0$, $K-$functional is defined as
\begin{equation*}
K_{2}(f,\delta )=\inf \bigl\{\|f-g\|+\eta \|g''\|:g\in W^{2}\bigr\},
\end{equation*}%
where norm-$\|.\|$ denotes the uniform norm on $C[0,1]$. Following the well
known inequality due to DeVore and Lorentz \cite{Devore}, there exists an
absolute constant $C>0$ such that
\begin{equation}
K_{2}(f,\delta )\le  C\omega _{2}(f,\sqrt{\delta }),  \label{3.3.9}
\end{equation}%
where the second order modulus of smoothness for $f\in C[0,1]$ is defined as
\begin{equation*}
\omega _{2}(f,\sqrt{\delta })=\sup\limits_{0<h\le  \sqrt{\delta }}\sup_{x,x+h\in
\lbrack 0,1]}|f(x+h)-f(x)|.
\end{equation*}%
The usual modulus of continuity for $f\in C[0,1]$ is defined as
\begin{equation*}
\omega (f,\delta )=\sup\limits_{0<h\le  \delta }\sup\limits_{x,x+h\in \lbrack
0,1]}|f(x+h)-f(x)|.
\end{equation*}

\begin{thm}\label{T-d1}
Let $n > 3$ be a natural number and let $0<q<p\le  1,$ $q_{0} =
q_{0}(n) \in (0,p)$ be defined as in Lemma \ref{l2}. Then, there exists an
absolute constant $C > 0$ such that
\begin{equation*}
\vert D_{n}^{p,q}(f,x) - f(x) \vert \le    C \,  \omega_{2}\left(f ,
[n+2]_{p,q}^{-1/2} \delta_{n}(x)\right) + \omega \left(f,\frac{2x}{%
[n+2]_{p,q}}\right),
\end{equation*}
where $f \in C[0,1],$ $\delta_{n}^{2}(x) = \varphi^{2}(x) +
\frac{1}{[n+2]_{p,q}}$, $\varphi^{2}(x)=x(1-x)$, $x \in [0,1]$ and
$q \in (q_{0},1)$.
\end{thm}

\begin{proof}
For $f \in C[0,1]$ we define
\begin{equation*}
\widetilde{D}_{n}^{p,q}(f,x) \ = \ D_{n}^{p,q}(f,x) + f(x) - f\left( \frac{%
p[n]_{p,q}x}{[n+2]_{p,q}} \right).
\end{equation*}

\noindent Then, by Lemma \ref{l1}, we immediately observe that
$$
\tilde{D}_{n}^{p,q}(1,x) = D_{n}^{p,q}(1,x) = 1$$ and $$
\tilde{D}_{n}^{p,q}(t,x) = D_{n}^{p,q}(t,x) + x - \frac{p[n]_{p,q}x}{%
[n+2]_{p,q}} = x.$$
By applying Taylor's formula
\begin{equation*}
g(t) = g(x) + (t-x)g'(x) + \int\limits_{x}^{t}(t-u) g''(u)\,{\D}u,
\end{equation*}
we get
\begin{eqnarray*}
\widetilde{D}_{n}^{p,q}(g,x) &=& g(x) + \widetilde{D}_{n}^{p,q}\left( \int\limits_{x}^{t} \
(t-u)g''(u)\, {\D}u , x \right)  \\
&=& g(x) + D_{n}^{p,q}\left( \int\limits_{x}^{t} (t-u) g''(u)\,
{\D}u , x \right)\\[2mm]
&&\qquad
- \int\limits_{x}^{\frac{p[n]_{p,q}x}{[n+2]_{p,q}}}
\left(\frac{p[n]_{p,q}x}{[n+2]_{p,q}} - u \right) g''(u)\,{\D}u.
\end{eqnarray*}

\noindent Thus,
\noindent$\displaystyle \vert \widetilde{D}_{n}^{p,q}(g,x) - g(x)
\vert $
\begin{eqnarray}  \label{3.3.12}
\!\!\!&\le &\!\!\! D_{n}^{p,q}\left( \Biggm \vert \int\limits_{x}^{t} \  \vert t-u \vert \cdot
\vert g''(u) \vert \ du \Biggm \vert , x \right) + \Biggm
\vert \int\limits_{x}^{\frac{p[n]_{p,q}x}{[n+2]_{p,q}}}  \Biggm \vert \frac{%
p[n]_{p,q}x}{[n+2]_{p,q}} - u \Biggm \vert \vert g''(u) \vert {\D}u \Biggm \vert  \notag \\[2mm]
\!\!\!&\le &\!\!\! D_{n}^{p,q}( (t-x)^{2} , x ) \Vert g''\Vert +
\left( \frac{p[n]_{p,q}x}{[n+2]_{p,q}} - x \right)^{2}\Vert
g''\Vert.
\end{eqnarray}

Also, we have
\begin{eqnarray} \label{ee}
&& D_{n}^{p,q}( (t-x)^{2} , x ) + \left( \frac{%
p[n]_{p,q}x}{[n+2]_{p,q}} - x \right)^{2} \notag \\
&&\qquad \le  \frac{6}{[n+2]_{p,q}} \left( \varphi^{2}(x) + \frac{1}{[n+2]_{p,q}}
\right) + \left(\frac{(p[n]_{p,q} -  [n+2]_{p,q} ) x}{[n+2]_{p,q}}
\right)^{2}\qquad \notag \\
&&\qquad \le  \frac{10}{[n+2]_{p,q}} \left( \varphi^{2}(x) + \frac{1}{[n+2]_{p,q}}
\right).
\end{eqnarray}

\noindent Hence, by (\ref{3.3.12}) and with the condition $n > 3$ and $x
\in [0,1],$ we have
\begin{equation}  \label{3.3.17}
\vert \widetilde{D}_{n}^{p,q}(g,x) - g(x) \vert \le  \frac{10}{[n+2]_{p,q}} \
\delta_{n}^{2}(x) \ \Vert g''\Vert.
\end{equation}
\noindent Furthermore, for $f\in C[0,1]$ we have $||D_{n}^{p,q}(f,x)|| \le
||f||$, thus
\begin{equation} \label{3.3.18}
\vert \widetilde{D}_{n}^{p,q}(f,x) \vert \le  \vert D_{n}^{p,q}(f,x) \vert +
\vert f(x) \vert + \Biggm \vert f\left( \frac{p[n]_{p,q} x}{%
[n+2]_{p,q}} \right) \Biggm \vert \le  3 \Vert f \Vert.
\end{equation}
\noindent for all $f \in C[0,1].$

Now, for $f \in C[0,1]$ and $g \in W^{2},$ we obtain

\noindent$\displaystyle \vert D_{n}^{p,q}(f,x) - f(x) \vert$
\begin{eqnarray}
&=&\!\!\! \Biggm \vert \widetilde{D}_{n}^{p,q}(f,x) - f(x) + f\left( \frac{p
[n]_{p,q} x}{[n+2]_{p,q}} \right) - f(x) \Biggm \vert  \notag \\
&\le &\!\!\! \vert \widetilde{D}_{n}^{p,q}(f-g,x) \vert + \vert \widetilde{D}%
_{n}^{p,q}(g,x) - g(x) \vert + \vert g(x) - f(x) \vert \\[2mm]
&&\qquad\qquad\qquad\qquad\qquad\quad + \Biggm \vert %
f\left( \frac{p[n]_{p,q} x}{[n+2]_{p,q}} \right) - f(x) \Biggm \vert
\notag \\
&\le &\!\!\! 4 \ \Vert f - g \Vert + \frac{10}{[n+2]_{p,q}}  \delta_{n}^{2}(x)
 \Vert g''\Vert + \omega \left( f , \Biggm \vert \frac{%
( p[n]_{p,q} - [n+2]_{p,q} ) x}{[n+2]_{p,q}} \Biggm \vert \right),
\notag \\  \notag
\end{eqnarray}
\noindent where we have used (\ref{3.3.17}) and (\ref{3.3.18}). Taking the
infimum on the right hand side over all $g \in W^{2},$ we obtain at once
\[
\vert D_{n}^{p,q}(f,x) - f(x) \vert \le  10  K_{2}\left( f , \frac{1}{%
[n+2]_{p,q}} \delta_{n}^{2}(x) \right) + \omega \left( f , \frac{2x}{%
[n+2]_{p,q}} \right).
\]

\noindent Finally, in view of (\ref{3.3.9}), we find
\[
|D_{n}^{p,q}(f,x)-f(x)| \le C\,  \omega _{2}\left(
f,[n+2]_{p,q}^{-1/2}\delta _{n}(x)\right) +\omega \left( f,\frac{2x}{%
[n+2]_{p,q}}\right).
\]

\noindent This completes the proof of the theorem.
\end{proof}

The weighted modulus of continuity of second order is defined as:
\begin{equation*}
\omega _{2}^{\varphi }(f,\sqrt{\delta }) = \sup\limits_{0<h\le  \sqrt{\delta }%
}\sup\limits_{x, x\pm h\varphi \in \lbrack 0,1]}|f(x+h\varphi (x))-2f(x)+f(x-h\varphi
(x))|
\end{equation*}%
\noindent where $\varphi
(x)=\sqrt{x(1-x)}$. The corresponding $K$-functional is defined by
\begin{equation*}
\overline{K}_{2,\varphi }(f,\delta ) = \inf \bigl\{ \Vert f-g\Vert +\delta
\Vert \varphi ^{2}g''\Vert+\delta^2
\Vert g''\Vert:g\in W^{2}(\varphi ) \bigr\}
\end{equation*}%
where
\begin{equation*}
W^{2}(\varphi )=\bigl\{g\in C[0,1]:g'\in AC_{\loc}[0,1],\varphi
^{2}g''\in C[0,1]\bigr\}
\end{equation*}%
and $g^{\prime }\in AC_{\loc}[0,1]$ means that $g$ is differentiable and
$g'$ is absolutely continuous on every closed interval $[a,b]\subset
\lbrack 0,1]$. It is well-known due to Ditzian-Totik (see \cite[p. 24,
Theorem 1.3.1]{Ditzian}) that
\begin{equation}
\overline{K}_{2,\varphi }(f,\delta ) \le   C \omega _{2}^{\varphi }(f,%
\sqrt{\delta })  \label{3.3.19}
\end{equation}%
\noindent for some absolute constant $C>0.$ Moreover, the Ditzian-Totik
moduli of first order is given by
\begin{equation*}
\vec{\omega}_{\psi }(f,\delta )\ =\  \sup\limits_{0<h\leq \delta }\sup\limits_{x,x\pm h\psi
(x)\in \lbrack 0,1]}|f(x+h\psi (x))-f(x)|,
\end{equation*}
\noindent where $\psi$ is an admissible step-weight function on $[0,1].$ \\

Now, we state our next main result, i.e., the global estimate.

\begin{thm}
\label{td2}Let $n>3$ be a natural number and let $0<q<p\le
1,$ $q_{0}=q_{0}(n)\in (0,p)$ be defined as in Lemma \ref{l2}. Then, there exists
an absolute constant $C>0$ such that
\begin{equation*}
\Vert D_{n}^{p,q}f-f\Vert   \le   C \omega _{2}^{\varphi
}(f,[n+2]_{q}^{-1/2}) + \vec{\omega}_{\psi }(f,[n+2]_{q}^{-1}),
\end{equation*}%
where $f\in C[0,1],$ $q\in (q_{0},1)$, and $\psi (x)=x,$ $x\in \lbrack
0,1]. $
\end{thm}

\begin{proof}
We again consider
\begin{equation*}
\widetilde{D}_{n}^{p,q}(f,x)\ =\ D_{n}^{p,q}(f,x)+f(x)-f\left( \frac{%
p[n]_{p,q}x}{[n+2]_{p,q}}\right) ,
\end{equation*}

\noindent where $f\in C[0,1].$ Now, using Taylor's formula, we have
\[
g(t)=g(x)+(t-x)\ g^{\prime }(x)+\int\limits_{x}^{t}\, (t-u)\, g''(u)\, {\D}u.
\]
Using (\ref{l1}), we obtain
\begin{eqnarray*}
\widetilde{D}_{n}^{p,q}(g,x)\!\!\!& = &\!\!\!g(x)+D_{n}^{p,q}\left( \int\limits_{x}^{t}\ (t-u)\,
g''(u)\, {\D}u,x\right) \\[2mm]
&&\qquad\qquad  -\int\limits_{x}^{\frac{p[n]_{p,q}x}{%
[n+2]_{p,q}}}\  \left( \frac{p[n]_{p,q}x}{[n+2]_{p,q}}-u\right)
g''(u)\, {\D}u.
\end{eqnarray*}
Thus,
\begin{eqnarray}\label{3.3.20}
|\tilde{D}_{n}^{p,q}(g,x)-g(x)|\!\!\!&\le&\!\!\! D_{n}^{p,q}\left( \Biggm \vert \int\limits_{x}^{t}|t-u|  |g''(u)|\,{\D}u\Biggm \vert,x\right)\nonumber \\ [2mm]	
&& +\Biggm \vert \int\limits_{x}^{\frac{%
p[n]_{p,q}x}{[n+2]_{p,q}}}\Biggm \vert \frac{p[n]_{p,q}x}{%
[n+2]_{p,q}}-u\Biggm \vert  |g''(u)|\,{\D}u\Biggm \vert .\qquad
\end{eqnarray}

Since $\delta_{n}^{2}$ is concave on $[0,1],$ therefore, for $u = t + \tau ( x - t ),$ $\tau \in [0,1],$ the following estimate holds:
\begin{equation*}
\frac{\vert t - u \vert}{\delta_{n}^{2}(u)} = \frac{\tau \vert x - t \vert}{%
\delta_{n}^{2}(t + \tau ( x - t ))} \le  \frac{\tau \vert x - t \vert}{%
\delta_{n}^{2}(t) + \tau ( \delta_{n}^{2}(x) - \delta_{n}^{2}(t) )} \le
\frac{\vert t - x \vert}{\delta_{n}^{2}(x)} \cdot
\end{equation*}

\noindent Thus, using (\ref{3.3.20}), we obtain

\noindent $\displaystyle|\tilde{D}_{n}^{p,q}(g,x)-g(x)|\  $
\begin{eqnarray}
\!\!\!&\le&\!\!\!D_{n}^{p,q}\left(\Biggm \vert\! \int\limits_{x}^{t} \frac{|t-u|}{\delta
_{n}^{2}(u)}\,{\D}u\Biggm \vert,x\right)  \Vert \delta _{n}^{2}g''\Vert +\Biggm \vert\!\! \int\limits_{x}^{\frac{p[n]_{p,q}x}{[n+2]_{p,q}}%
} \frac{\Bigm \vert \frac{p[n]_{p,q}x}{[n+2]_{p,q}}-u\Bigm \vert}{%
\delta _{n}^{2}(u)}\,{\D}u\!\!\Biggm \vert \! \Vert \delta _{n}^{2}g''\Vert  \notag \\
\!\!\!&\le  &\!\!\!\frac{1}{\delta _{n}^{2}(x)} D_{n}^{p,q}((t-x)^{2},x) \Vert
\delta _{n}^{2}g''\Vert +\frac{1}{\delta _{n}^{2}(x)}
\left( \frac{p[n]_{p,q}x}{[n+2]_{p,q}}-x\right) ^{2}\Vert \delta
_{n}^{2}g''\Vert . \notag
\end{eqnarray}

\noindent For $x\in \lbrack 0,1],$ in view of (\ref{ee}) and
\begin{equation*}
\delta _{n}^{2}(x)\ |g''(x)| = |\varphi^{2}(x) \, g''(x)|+\frac{1%
}{[n+2]_{p,q}}\ |g''(x)| \le  \Vert \varphi^{2} \, g''\Vert +\frac{1}{%
[n+2]_{p,q}}\ \Vert g''\Vert ,
\end{equation*}

\noindent we have
\begin{equation}
|\widetilde{D}_{n}^{p,q}(g,x)-g(x)| \le  \frac{5}{[n+2]_{p,q}}\left( \Vert
\varphi ^{2}g''\Vert +\frac{1}{[n+2]_{p,q}}\ \Vert
g''\Vert \right) . \label{3.3.21}
\end{equation}

Using the fact that $[n]_{p,q} \le  [n+2]_{p,q}$, (\ref{3.3.18}) and (\ref{3.3.21}), for $f \in C[0,1],$ we get
\begin{eqnarray*}
|D_{n}^{p,q}(f,x)-f(x)|\!\!\!&\le&\!\!\!	
|\tilde{D}_{n}^{p,q}(f-g,x)|+|\tilde{D}%
_{n}^{p,q}(g,x)-g(x)|+|g(x)-f(x)|\\[2mm]
&&\qquad\qquad\qquad\ \, +\Biggm \vert f\left( \frac{%
p[n]_{p,q}x}{[n+2]_{p,q}}\right) -f(x)\Biggm \vert \\
\!\!\!&\le&\!\!\!4\Vert f-g\Vert +\frac{10}{[n+2]_{p,q}}\Vert \varphi
^{2}g''\Vert
 +\frac{10}{[n+2]_{p,q}}\Vert g''\Vert\\[2mm]
&&\qquad\qquad\qquad\ \,+\Biggm \vert f\left( \frac{p[n]_{p,q}x}{[n+2]_{p,q}}%
\right) -f(x)\Biggm \vert .
\end{eqnarray*}
On taking the infimum on the right hand side over all $g\in
W^{2}(\varphi ),$ we obtain
{\small\begin{equation}\label{3.3.22}
|D_{n}^{p,q}(f,x)-f(x)| \le  10  \overline{K}_{2,\varphi}\left( f,\frac{1}{%
[n+2]_{p,q}}\right) +\Biggm \vert f\left( \frac{p[n]_{p,q}x}{%
[n+2]_{p,q}}\right) -f(x)\Biggm \vert .
\end{equation}}
Moreover,
{\small
$$\Biggm \vert f\left( \frac{p[n]_{p,q}x}{
[n+2]_{p,q}}\right) -f(x)\Biggm \vert =\Biggm \vert f\left( x+\psi (x)\ \frac{x \, \left(p[n]_{p,q}-[n+2]_{p,q}\right)
}{[n+2]_{p,q} \ \psi (x)}\right) -f(x)\Biggm \vert\qquad\qquad$$}
{\small\begin{eqnarray*}
\!\!\!&\le&\!\!\! \sup\limits_{t,t+\psi (t)\
\, \frac{x \left(p[n]_{p,q}-[n+2]_{p,q}\right)
}{[n+2]_{p,q} \ \psi (x)}\in \lbrack 0,1]}\  \Biggm \vert %
f\left( t+\psi (t)  \frac{x \, \left(p[n]_{p,q}-[n+2]_{p,q}\right)}{%
[n+2]_{p,q}  \psi (x)}\right) -f(t)\Biggm \vert  \\
\!\!\!&\le&\!\!\! \vec{\omega}_{\psi }\left( f,\frac{|x \, \left(p[n]_{p,q}-[n+2]_{p,q}\right)|}{%
[n+2]_{p,q}   \psi (x)}\right)\\[2mm]
\!\!\!&\le&\!\!\!  \vec{\omega}_{\psi }\left( f,\frac{x}{%
[n+2]_{p,q}   \psi (x)}\right) =\vec{\omega}_{\psi }\left( f,\frac{1}{%
[n+2]_{p,q}}\right) .
\end{eqnarray*}}

Hence, by (\ref{3.3.22}) and (\ref{3.3.19}), we finally get
\begin{equation*}
\Vert D_{n}^{p,q}f-f\Vert \  \le  \ C\  \omega _{2}^{\varphi
}(f,[n+2]_{p,q}^{-1/2})\ +\  \omega _{\psi }(f,[n+2]_{p,q}^{-1}).
\end{equation*}%
 This completes the proof of the theorem.
\end{proof}

\begin{rem}
For $q\in (0,1)$ and $p\in (q,1]$ it is obvious that $\lim\limits_{n\rightarrow
\infty }\left[ n\right] _{p,q}=\frac{1}{p-q}.$ An example of such choice for $p,q$ depending on $n$ is recently given in \cite{VA}.
\end{rem}

\begin{ex}
 Now, we show comparisons and some illustrative
graphs for the convergence of $(p,q)$-analogue of Bernstein-Durrmeyer
operators $D_{n}^{p,q}(f,x)$ for different values of the parameters $p$ and $q,$ such that $0< q < p \le  1$.

For $x\in [0,1]$, $p = 0.5$ and $q = 0.4$, the convergence of the difference of the operators $D_{n}^{p,q}(f,x)$ to the function $f$, where $f(x)=9x^2-4x+5$, for different values of $n$ is illustrated in Fig.~\ref{Fig1}.
\begin{figure}[h]
\includegraphics[width=0.8\textwidth]{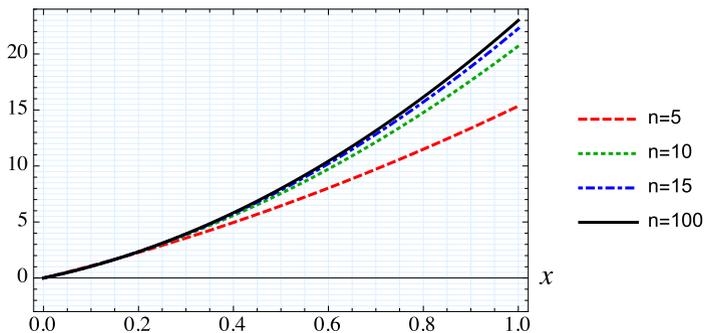}
\caption[]{Graphics of the difference $x\mapsto D_{n}^{0.5,0.4}(f,x)-f(x)$ for $x\in[0,1]$, when $f(x)=9x^2-4x+5$ and $n=5,10,15$ and $n=100$}\label{Fig1}
\end{figure}
\begin{figure}[h]
\includegraphics[width=0.8\textwidth]{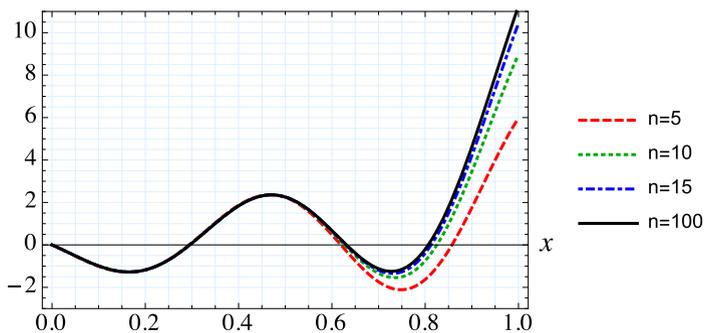}
\caption[]{Graphics of the difference $x\mapsto D_{n}^{0.5,0.4}(f,x)-f(x)$ for $x\in[0,1]$, when $f(x)=(x+1)^2\sin(10\pi x/3)$ and $n=5,10,15$ and $n=100$}\label{Fig2}
\end{figure}

The convergence of the difference of the operators $D_{n}^{p,q}(f,x)$ to the function $f$, where $f(x)=(x+1)^2\sin\left(\frac{10}3\pi x\right)$ for different values of $n$ and $x\in [0,1]$ is illustrated in Fig.~\ref{Fig2}.
\end{ex}

\begin{ex}
For the  function $f(x)=9 x^2 - 4 x + 5$ (and $p=0.5$, $q=0.4$), the limit of $D_n^{0.5,0.4}(f,x)$, when $n\to+\infty$,  is $f^*(x)=5 + (124/25) x + (576/25) x^2$. Graphics of $D_n^{0.5,0.4}(f,x)-f^*(x)$,   for $n=10,15,20,50$, are presented in Fig.~\ref{Fig3}.
\begin{figure}[h]
\includegraphics[width=0.8\textwidth]{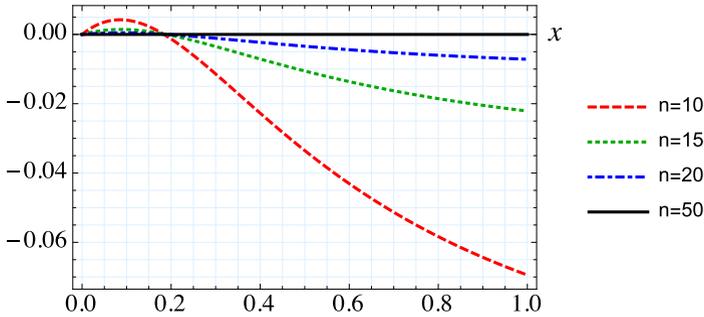}
\caption[]{Graphics of   $D_{n}^{0.5,0.4}(f,x)-f^*(x)$, when
$f(x)=9 x^2 - 4 x + 5$ and  $f^*(x)=5 + (124/25) x + (576/25) x^2$, for $n=10,15,20,50$}\label{Fig3}
\end{figure}
\end{ex}

\section{Better Approximation}

About a decade ago, King \cite{king} proposed a technique to obtain better approximation for the well known Bernstein polynomials. In this technique, these operators approximate each continuous function $f\in [0,1],$ while preserving the function $e_{2}(x)=x^2.$ These were basically compared with estimates of approximation by Bernstein polynomials. Various standard linear positive operators preserve $e_0$ and $e_1,$ i.e., preserve constant and linear functions, but this approach helps in reproducing the quadratic functions as well.

So, using King's technique, we modify the operators (\ref{operator}) as follows:
\begin{eqnarray*}
D_{n,p,q}^{*}(f,x)&=& \left[ n+1\right] _{p,q}%
\sum\limits_{k=1}^{n}p^{-(n-k+1)(n+k)/2}  b_{n,k}^{p,q}(1,r_{n}(x))\\
&&\qquad\times\,\int\limits_{0}^{1}b_{n,k-1}^{p,q}(t) f\left( t\right) {\D}_{p,q}t +  b_{n,0}^{p,q}(1,r_{n}(x)) f(0),
\end{eqnarray*}
where $r_{n}(x)= \frac{[n+2]_{p,q} \, x}{p [n]_{p,q}}$ and $x \in I_{n,p,q}=\left[0, \frac{[n+2]_{p,q}}{p [n]_{p,q}}\right]$.
Then, we have
\begin{eqnarray*}
D_{n,p,q}^{*}(e_0,x)= 1, \ \ D_{n,p,q}^{*}(e_1,x)= x,
\end{eqnarray*}

\begin{eqnarray*}
D_{n,p,q}^{*}(e_2,x)= \frac{(p+q)\, p^{n} \, x}{[n+3]_{p,q}} + \frac{([n]_{p,q}-p^{n-1}) \, q \, [n+2]_{p,q} \, x^{2}}{[n]_{p,q} [n+3]_{p,q}} \cdot
\end{eqnarray*}

Now, Theorem \ref{T-d1} can be modified as:
\begin{thm}
Let $n > 3$ be a natural number and let $0<q<p\le  1,$ $q_{0} =
q_{0}(n) \in (0,p)$ be defined as in Lemma \ref{l2}. Then, there exists an
absolute constant $C > 0$ such that
\begin{equation*}
\vert D_{n,p,q}^{*}(f,x) - f(x) \vert   \le    C \,  \omega_{2}\left(f ,
\sqrt{\delta_{n}^{p,q}(x)}\right) ,
\end{equation*}
\noindent where $x \in I_{n,p,q}=\left[0, \frac{[n+2]_{p,q}}{p [n]_{p,q}}\right],$ $q \in (q_{0},1)$, and
\begin{eqnarray*}
\delta_{n}^{p,q}(x)\!\!\!&=&\!\!\! D_{n,p,q}^{*}((t-x)^{2},x)\\
\!\!\!&=&\!\!\! \frac{(p+q)\, p^{n}  x}{[n+3]_{p,q}} + \frac{\{([n]_{p,q}-p^{n-1}) \, q \, [n+2]_{p,q} - [n]_{p,q} [n+3]_{p,q}\}x^{2}}{[n]_{p,q} [n+3]_{p,q}} \cdot
\end{eqnarray*}
\end{thm}

The proof is on similar lines, so we omit the details.

\begin{ex}
 We compare the convergence of $(p,q)$-analogue of Bernstein-Durr\-meyer
operators $D_{n}^{p,q}(f,x)$ with the operators $D_{n,p,q}^{*}(f,x).$ We have considered the same function as in the previous example.

For $x\in \left[0, \frac{[n+2]_{p,q}}{p [n]_{p,q}}\right],$ $p = 0.5$ and $q = 0.4$, the convergence of the difference of the operators $D_{n,p,q}^{*}(f,x)$ to the function $f$, where $f(x)=9x^2-4x+5$, for different values of $n$ is illustrated in Fig.~\ref{Fig4}.

\begin{figure}[h]
\includegraphics[width=0.8\textwidth]{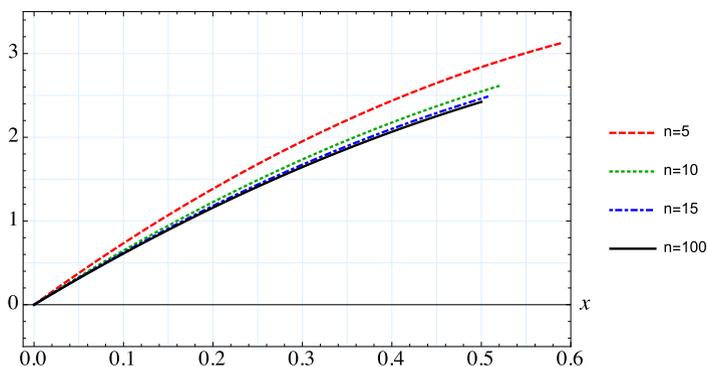}
\caption[]{Graphics of the difference $x\mapsto D^*_{n,0.5,0.4}(f,x)-f(x)$ for $x\in\left[0, {[n+2]_{p,q}}/{p [n]_{p,q}}\right]$, when $f(x)=9x^2-4x+5$ and $n=5,10,15$ and $n=100$.}\label{Fig4}
\end{figure}
\end{ex}


\end{document}